\documentclass[12pt,reqno]{amsart}

\newcommand\version{May 4, 2022}

\usepackage{amsmath,amsfonts,amsthm,amssymb,amsxtra}
\usepackage{bbm} % to get \1

\newtheorem{theorem}{Theorem}%[section]
\newtheorem{proposition}[theorem]{Proposition}

\newtheorem{corollary}[theorem]{Corollary}

\theoremstyle{definition}

\theoremstyle{remark}

\newcommand{\1}{\mathbbm{1}}

\newcommand{\C}{\mathbb{C}}

\renewcommand{\epsilon}{\varepsilon}

\newcommand{\loc}{{\rm loc}}

\newcommand{\N}{\mathbb{N}}

\renewcommand{\phi}{\varphi}
\newcommand{\R}{\mathbb{R}}

\DeclareMathOperator{\dist}{dist}

\DeclareMathOperator{\per}{Per}

\DeclareMathOperator{\re}{Re}

\DeclareMathOperator{\supp}{supp}

\DeclareMathOperator{\Tr}{Tr}

\begin{document}

\title[Rearrangement methods in the work of Elliott Lieb --- \version]{Rearrangement methods\\ in the work of Elliott Lieb}

\author{Rupert L. Frank}
\address[Rupert L. Frank]{Mathe\-matisches Institut, Ludwig-Maximilans Universit\"at M\"unchen, The\-resienstr.~39, 80333 M\"unchen, Germany, and Munich Center for Quantum Science and Technology, Schel\-ling\-str.~4, 80799 M\"unchen, Germany, and Mathematics 253-37, Caltech, Pasa\-de\-na, CA 91125, USA}
\email{r.frank@lmu.de}

\thanks{\copyright\, 2022 by the author. This paper may be reproduced, in its entirety, for non-commercial purposes.\\
	It is a great pleasure to thank Elliott Lieb for many interesting discussions, mathematical and otherwise, throughout the years. Helpful comments on this manuscript by Eric Carlen, Mathieu Lewin, Michael Loss and Julien Sabin are gratefully acknowledged.\\
	Partial support through US National Science Foundation grants DMS-1954995, as well as through the Deutsche Forschungsgemeinschaft (DFG, German Research Foundation) through Germany’s Excellence Strategy EXC-2111-390814868 is acknowledged.}

\dedicatory{Dedicated, in great admiration, to Elliott Lieb on the occasion of his 90th birthday}

\begin{abstract}
	We review some topics in the theory of symmetric decreasing rearrangements with a particular focus on Lieb's fundamental contributions. Topics covered include the Brascamp--Lieb--Luttinger theorem, the sharp Young and Hardy--Littlewood--Sobolev inequalities, as well as the continuity of the rearrangement map on Sobolev spaces.
\end{abstract}

\maketitle

Symmetrization of sets goes back to the nineteenth century works of Steiner, and later Schwarz, on the isoperimetric problem. In the 1920s Faber and Krahn introduced the symmetric decreasing rearrangement of functions in their proofs of Rayleigh's conjecture on the fundamental frequency of a membrane. This technique is discussed in the influential books by Hardy, Littlewood and P\'olya \cite{HLP} and by P\'olya and Szeg\H{o} \cite{PoSz}, first published, respectively, in 1934 and 1951.

Starting from the early 1970s, Lieb has made several fundamental contributions to this theory and our goal here is to review some of them, as well as some ramifications. 

The textbook \cite{LiLo} by Lieb and Loss, which has become a standard reference for symmetric decreasing rearrangement, is an excellent starting point for getting acquainted with these methods. Another useful reference, besides the three mentioned already, is \cite{Ba}.

This paper is organized as follows. In Section \ref{sec:basic}, we quickly review the basic properties of rearrangement. This is followed in Section \ref{sec:bll} with a rather detailed discussion of the Riesz rearrangement inequality and its far-reaching generalization, the Brascamp--Lieb--Luttinger theorem. Items that we discuss here include the Brascamp--Lieb proof of the sharp Young inequality and Lieb's proof of the sharp Hardy--Littlewood--Sobolev inequality. The following two sections, Sections \ref{sec:fracsob} and \ref{sec:sob}, concern symmetrization in fractional and first-order Sobolev spaces, respectively. We present both Lieb's simple proofs of the corresponding rearrangement inequalities, as well as his deep results, jointly with Almgren, on continuity of the rearrangement map. The final section, Section \ref{sec:stab}, concerns stability questions and contains, among other things, a recent quantitative version of the Riesz rearrangement inequality.

%%%%%%%%%%%%%%%%%%%%%%%%%

\section{Basic properties of rearrangement}\label{sec:basic}

We briefly recall the definition and review basic properties of symmetric decreasing rearrangement. For proofs, we refer to the textbook \cite{LiLo}.

Let $A\subset\R^d$ be a set of finite measure. (Here, and in what follows, when we assume that a set has finite measure, we assume implicitly that the set is (Borel) measurable.) We define its symmetrization $A^*$ to be the open ball in $\R^d$ centered at the origin with the same measure as $A$, that is,
$$
A^* := \left\{ x\in\R^d :\ |x| < \omega_d^{-1/d} |A|^{1/d} \right\},
$$
where $\omega_d:=|\{x\in\R^d:\ |x|<1\}|$ is the measure of the unit ball. (The choice of an \emph{open} ball is for definiteness only. The theory can be developed similarly with a closed ball.) 

We say that a function $f$ on $\R^d$ is \emph{vanishing at infinity} if it is measurable and
\begin{equation}
	\label{eq:vanish}
	|\{ |f|>\tau \}|<\infty
	\qquad\text{for all}\ \tau>0 \,.
\end{equation}
The \emph{symmetric decreasing rearrangement} of a function $f$ on $\R^d$, vanishing at infinity, is defined by
\begin{equation}
	\label{eq:schwarz}
	f^*(x) := \int_0^\infty \1_{\{|f|>\tau\}^*}(x) \,d\tau
	\qquad\text{for all}\ x\in\R^d \,.
\end{equation}
This formula should be compared with the `layer cake representation' of $|f|$,
\begin{equation}
	\label{eq:layercake}
	|f(x)| = \int_0^{|f(x)|} \,d\tau = \int_0^\infty \1_{\{|f|>\tau\}}(x) \,d\tau
	\qquad\text{for all}\ x\in\R^d \,.
\end{equation}
By definition, $f^*$ is a nonnegative, radially symmetric and nonincreasing function. The radial symmetry and monotonicity imply that $f$ is measurable. It is an easy exercise to show that
\begin{equation}
	\label{eq:schwarzsuperlevel}
	\{ f^*>\tau\} = \{|f|>\tau\}^*
	\qquad\text{for all}\ \tau\geq 0 \,.
\end{equation}
This has several useful consequences. First, $f^*$ is lower semi-continuous, i.e., $\{f^*>\tau\}$ is open for any $\tau>0$. (If we had defined $A^*$ to be a \emph{closed} ball, $f^*$ would be upper semi-continuous.) Second, $f^*$ is a rearrangement of $f$, in the sense that $|\{|f^*|>\tau\}| = |\{|f|>\tau\}|$ for all $\tau\geq 0$. This implies, in particular, that
\begin{equation}
	\label{eq:symmp}
	\int_{\R^d} |f^*|^p \,dx = \int_{\R^d} |f|^p \,dx
	\qquad\text{for all}\ p>0 \,.
\end{equation}
A simple, but useful property of rearrangement is that, for any nonnegative functions $f$ and $g$ on $\R^d$, vanishing at infinity, one has
\begin{equation}
	\label{eq:simplest}
	\int_{\R^d} fg \,dx \leq \int_{\R^d} f^*g^* \,dx \,.
\end{equation}
The following theorem is a vast generalization of this. It appears in \cite{CZR}, but at least special cases seem to have been known before, as discussed in the introduction of that paper; see also \cite{Lo}.

\begin{theorem}\label{bl0}
	Let $F:[0,\infty)\times[0,\infty)\to [0,\infty)$ be continuous with $F(0,0)=0$ and
	$$
	F(u_2,v_2) + F(u_1,v_1) \geq F(u_2,v_1) + F(u_1,v_2)
	\qquad\text{for all}\ u_2\geq u_1\geq 0 \,,\ v_2\geq v_1\geq 0 \,.
	$$
	Then, for all nonnegative functions $f,g$ on $\R^d$, vanishing at infinity,
	\begin{equation}
		\label{eq:blsimple}
		\int_{\R^d} F(f(x),g(x))\,dx \leq \int_{\R^d} F(f^*(x),g^*(x))\,dx \,.
	\end{equation}
\end{theorem}

For twice continuously differentiable $F$, the assumed inequality is equivalent to $\frac{\partial^2}{\partial u\partial v}F\geq 0$.

For $F(u,v)=uv$, inequality \eqref{eq:blsimple} reduces to \eqref{eq:simplest}. Conversely, the proof of Theorem \ref{bl0} proceeds by reducing it to \eqref{eq:simplest}. Indeed, assuming, via an approximation argument, that $F$ is twice continuously differentiable, we can write
$$
F(u,v) = F(u,0) + F(0,v) + \iint_{(0,\infty)\times (0,\infty)} \frac{\partial^2 F}{\partial u\partial v}(s,t) \1_{(0,\infty)}(u-s) \1_{(0,\infty)}(v-t)\,ds\,dt \,.
$$
The contributions of the first two terms on the right side to $\int_{\R^d} F(f(x),g(x))\,dx$ are invariant under rearrangement, while, for fixed $s,t\in(0,\infty)$, we have, by \eqref{eq:simplest},
$$
\int_{\R^d} \1_{\{f>s\}}(x) \1_{\{g>t\}}(x)\,dx \leq \int_{\R^d} \1_{\{f^*>s\}}(x) \1_{\{g^*>t\}}(x)\,dx \,.
$$
This proves \eqref{eq:blsimple}.

Let us draw some conclusions from Theorem \ref{bl0}. Taking $F(u,v) = j(u)+j(v)-j(|u-v|)$, which satisfies the assumption provided $j$ is nonnegative and convex on $[0,\infty)$ with $j(0)=0$, we obtain the following rearrangement inequalities.

\begin{corollary}\label{expand}
	Let $j$ be a nonnegative, convex function on $[0,\infty)$ with $j(0)=0$. Then, for any nonnegative functions in $f$ and $g$ on $\R^d$, vanishing at infinity,
	$$
	\int_{\R^d} j(|f^*(x)-g^*(x)|)\,dx \leq \int_{\R^d} j(|f(x)-g(x)|)\,dx
	$$
	and
	$$
	\int_{\R^d} j(f^*(x)+g^*(x))\,dx \geq \int_{\R^d} j(f(x)+g(x))\,dx \,.
	$$
\end{corollary}

This corollary appears in \cite{Chi}; see also \cite{CrTa} for the first inequality.

Further specializing to $j(t)=t^p$ with $p\geq 1$ and using $|f(x)-g(x)|\geq ||f(x)|-|g(x)||$ and $|f(x)+g(x)|\leq |f(x)|+|g(x)|$, we find that for any complex-valued functions $f,g$ on $\R^d$, vanishing at infinity,
\begin{equation}
	\label{eq:expand}
	\| f - g \|_{L^p(\R^d)} \geq \|f^*-g^*\|_{L^p(\R^d)} \,,
	\qquad
	\| f + g \|_{L^p(\R^d)} \leq \|f^* + g^*\|_{L^p(\R^d)} \,.
\end{equation}
The first inequality here means that rearrangement is nonexpansive on $L^p(\R^d)$.

For the sum of the two terms in the inequalities \eqref{eq:expand} and for any complex-valued functions $f,g$ on $\R^d$, vanishing at infinity, Carlen and Lieb \cite{CaLi} observed that
\begin{equation}
	\label{eq:hanner}
	\| f - g\|_{L^p(\R^d)}^p + \|f+g\|_{L^p(\R^d)}^p \geq \| f^* - g^*\|_{L^p(\R^d)}^p + \|f^*+g^*\|_{L^p(\R^d)}^p
\end{equation}
provided $1\leq p\leq 2$. For $2\leq p<\infty$ the inequality reverses. For nonnegative $f,g$ this is a consequence of Theorem \ref{bl0}, and the general case can be reduced to this. Inequality \eqref{eq:hanner} is the commutative analogue (and generalization) of a matrix re\-arrangement inequality in \cite{CaLi}, which is, unfortunately, outside of the scope of this review.

\medskip

For a generalization of Theorem \ref{bl0} to more than two functions we refer to \cite{Lo,Br,BuHa}. The latter paper also discusses cases of equality.

\medskip

We emphasize that in this section the fact that the underlying space is $\R^d$ played a relatively minor role. For a function on a general measure space, one can define a rearrangement to be an nonincreasing, equimeasurable function on a subinterval of $[0,\infty)$ of length equal to the total measure of the given space; see \cite[Chapter 1]{Ba}. The structure of $\R^d$ will play an important role in the following sections.

%%%%%%%%%%%%%%%%%%%%%%%%%%%%

\section{The Brascamp--Lieb--Luttinger theorem}\label{sec:bll}

A fundamental theorem about rearrangements asserts that, for nonnegative functions $f,g,h$ on $\R^d$, vanishing at infinity, we have
\begin{equation}
	\label{eq:riesz}
	\iint_{\R^d\times\R^d} f(x)g(x-y)h(y)\,dx\,dy \leq \iint_{\R^d\times\R^d} f^*(x)g^*(x-y)h^*(y)\,dx\,dy \,.
\end{equation}
This is due to Riesz \cite{Ri} in the one-dimensional case and was formulated by Sobolev \cite{So} in the multidimensional case.\footnote{It is pointed out in \cite{Bu} that Sobolev's proof is incomplete.}

The following theorem, due to Brascamp, Lieb and Luttinger \cite{BLL}\footnote{Essentially the same theorem appears in the paper \cite{Ro2} by Rogers; see also his earlier work \cite{Ro1}. It is pointed out in \cite[p.\ 2]{ChON} that the proof in \cite{Ro2} may be incomplete. Rogers' work has been rediscovered only recently. It has been suggested to call Theorem \ref{bll} the Rogers--Brascamp--Lieb--Luttinger theorem, but here we stick to its traditional appellation.}, is a vast generalization of the Riesz rearrangement theorem. To state it, let $M\leq N$ and consider $(b_{n,m})\in \R^{N\times M}$. For nonnegative functions $f_1,\ldots,f_N$ on $\R^d$, we set
$$
I[f_1,\ldots,f_N]
:= \int_{\R^d}\cdots \int_{\R^d} \prod_{n=1}^N f_n\Big(\sum_{m=1}^M b_{n,m} x_m\Big)\,dx_1\cdots\,dx_M \,. 
$$

\begin{theorem}[Brascamp--Lieb--Luttinger theorem]\label{bll}
	Let $f_1,\ldots,f_N$ be nonnegative functions on $\R^d$, vanishing at infinity. Then
	$$
	I[f_1,\ldots,f_N] \leq I[f_1^*,\ldots,f_N^*] \,.
	$$
\end{theorem}

For $M=1$, $N=2$ and $(b_{n,m})=\begin{pmatrix} 1 \\ 1 \end{pmatrix}$ the theorem reduces to the simple rearrangement inequality \eqref{eq:simplest}. More interestingly, for $M=2$, $N=3$ and 
$$
(b_{n,m}) = \begin{pmatrix} 1 & 0 \\ 1 & - 1 \\ 0 & 1 \end{pmatrix}
$$
it becomes the Riesz rearrangement inequality \eqref{eq:riesz}.

The heart of the proof of Theorem \ref{bll} is the case $d=1$ and each function $f_n$ equal to the characteristic function of a finite number of non-overlapping intervals. Brascamp, Lieb and Luttinger deal with this case by a beautiful `flow' argument, using the Brunn--Minkowski inequality. The full result for $d=1$ follows then by relatively soft arguments. The higher-dimensional case is deduced from the one-dimensional case by repeated Steiner symmetrization. At this point Brascamp, Lieb and Luttinger provide a rigorous proof of the fact that Steiner-symmetrizing a set repeatedly in well-chosen directions leads to a sequence converging (in measure) to a ball of the same measure. For the details of the proof of Theorem \ref{bll} we refer to the original paper \cite{BLL}.

%%%%%%%%%%%%%%%%%%%%%%%%%

\subsection{Applications of the BLL theorem}

We now present six applications of Theorem \ref{bll}. The first four use only the simpler version of the Riesz rearrangement inequality \eqref{eq:riesz} (in fact, only its version where one of the functions is already symmetric decreasing), while the last two use the full power of the BLL theorem.

\subsubsection*{Application 1: The isoperimetric inequality}

It is well known that one can deduce the isoperimetric inequality from the Riesz rearrangement inequality. The author learned this argument from Lieb. Our presentation here uses some ideas from \cite{Si}. For a measurable set $A\subset\R^d$, we put
$$
s(A) := \liminf_{\epsilon\to 0^+} \epsilon^{-1} \left|\{ x\in A:\ \dist(x,\R^d\setminus A)<\epsilon\}\right|.
$$
This is one of several notions of Minkowski content. Clearly, for $A$ with sufficiently regular boundary (Lipschitz, say), $s(A)$ coincides with the surface area of $\partial A$.

\begin{corollary}\label{isoper}
	If $A\subset\R^d$ has finite measure, then $s(A)\geq s(A^*)$.
\end{corollary}

Thus, among all sets of a given measure, the ball has the smallest surface area, measured by $s(\cdot)$. This corollary implies also the isoperimetric inequality in the framework of sets of finite perimeter, since for any such set there is a sequence $(A_n)$ of sets with regular boundary whose perimeters converge to that of the given set and their perimeters coincide with $s(A_n)$; see, for instance, \cite{Ma}.

\begin{proof}
	We will prove that, for any $\epsilon>0$,
	\begin{equation}
		\label{eq:concentration}
		\left|\{\dist(\cdot\,,\R^d\setminus A)\geq\epsilon\}\right| \leq
		\left|\{\dist(\cdot\,,\R^d\setminus A^*)\geq\epsilon\}\right|.
	\end{equation}
	Subtracting this from $|A|=|A^*|$, dividing by $\epsilon$ and letting $\epsilon\to 0^+$, we obtain the claimed inequality.
	
	First, note that \eqref{eq:riesz} implies that for all $1\leq p\leq\infty$,
	\begin{equation}
		\label{eq:rieszconv}
		\|g*h\|_{L^p(A)} \leq \|g^**h^*\|_{L^p(A^*)} \,.
	\end{equation}
	Indeed, by \eqref{eq:riesz} and H\"older's inequality,
	\begin{align*}
		\left| \int_{A} f (g*h)\,dx \right| &  \leq \int_{A^*} f^* (g^**h^*)\,dx \leq \| f^* \|_{L^{p'}(A^*)} \left\|g^**h^* \right\|_{L^p(A^*)} \\
		& = \| f \|_{L^{p'}(A)} \|g^**h^*\|_{L^p(A^*)}  \,.
	\end{align*}
	
	Now, since $|B_\epsilon|^{-1}\1_{B_\epsilon} * \1_A = 1$ on $\{\dist(\cdot\,,\R^d\setminus A)\geq\epsilon\}$, we have for any $1\leq p<\infty$, by \eqref{eq:rieszconv},
	\begin{align*}
		\left|\{\dist(\cdot\,,\R^d\setminus A)\geq\epsilon\}\right| & \leq \left| \{ |B_\epsilon|^{-1}\1_{B_\epsilon} * \1_A \geq 1\} \cap A \right| \\
		& \leq \left\| |B_\epsilon|^{-1}\1_{B_\epsilon} * \1_A \right\|_{L^p(A)}^p \\
		& \leq \left\| |B_\epsilon|^{-1}\1_{B_\epsilon} * \1_{A^*} \right\|_{L^p(A^*)}^p \,.
	\end{align*}
	Since $|B_\epsilon|^{-1}\1_{B_\epsilon} * \1_{A^*}<1$ for $||x|-r_A|>\epsilon$ and $=1$ for $||x|-r_A|\leq\epsilon$, where $r_A = (|A|/\omega_d)^{1/d}$, we easily deduce, by dominated convergence, that
	$$
	\lim_{p\to\infty} \left\| |B_\epsilon|^{-1}\1_{B_\epsilon} * \1_{A^*} \right\|_{L^p(A^*)}^p 
	= \left|\{\dist(\cdot\,,\R^d\setminus A^*)\geq\epsilon\}\right| \,.
	$$
	This proves \eqref{eq:concentration}.
\end{proof}

Alternatively, one can first deduce the Brunn--Minkowski inequality (in the version of Brascamp and Lieb \cite{BrLi2}) from the Riesz rearrangement inequality (see \cite[Section 8.4]{Ba}) and then the isoperimetric inequality from the Brunn--Minkowski inequality. This is not circular since, although the Brunn--Minkowski inequality is used in the proof of the BLL theorem, in the special case of the Riesz rearrangement inequality \eqref{eq:riesz} its use can be avoided.

\subsubsection*{Application 2: A rearrangement inequality for the gradient}

Using the Riesz rearrangement inequality, Lieb \cite{LCho} gave a rather simple proof of the fact that the $L^2$-norm of the gradient of a function does not increase under rearrangement. Later, together with Almgren \cite{AlLi}, he showed that the Riesz rearrangement inequality also implies the corresponding fact for `fractional generalizations' of the $L^2$-norm of the gradient. We will present the proofs of these two important theorems in Sections \ref{sec:sob} and \ref{sec:fracsob}, respectively.

\subsubsection*{Application 3: An extension of Riesz's rearrangement inequality}

Berestycki and Lieb (see \cite[Theorem 2.2]{AlLi}) prove the following powerful extension of \eqref{eq:riesz}, which will be useful later on.

\begin{theorem}\label{bl}
	Let $F:[0,\infty)\times[0,\infty)\to [0,\infty)$ be continuous with $F(0,0)=0$ and
	$$
	F(u_2,v_2) + F(u_1,v_1) \geq F(u_2,v_1) + F(u_1,v_2)
	\qquad\text{for all}\ u_2\geq u_1\geq 0 \,,\ v_2\geq v_1\geq 0 \,,
	$$
	and let $0\leq W\in L^1(\R^d)$ and $a,b\in\R\setminus\{0\}$. Then, for all nonnegative functions $f,g$ on $\R^d$, vanishing at infinity,
	\begin{equation}
		\label{eq:bl}
		\iint_{\R^d\times\R^d} \!\!\! F(f(x),g(y))\,W(ax+by) \,dx\,dy \leq \iint_{\R^d\times\R^d} \!\!\! F(f^*(x),g^*(y))\,W^*(ax+by)\,dx\,dy \,.
	\end{equation}
\end{theorem}

The condition on $F$ is the same as in Theorem \ref{bl0}. Recall that, for twice continuously differentiable $F$, it is equivalent to $\frac{\partial^2}{\partial u\partial v}F\geq 0$.

The strategy of proof is also similar to that in Theorem \ref{bl0}, except that the role of the simplest rearrangement inequality \eqref{eq:simplest} is now played by the Riesz rearrangement inequality \eqref{eq:riesz}. Indeed, for $F(u,v)=uv$, $a=1=-b$, inequality  \eqref{eq:bl} reduces to \eqref{eq:riesz}. Conversely, using the integral representation of $F$ similarly as before, we can deduce \eqref{eq:bl} from \eqref{eq:riesz}.

We also note that Theorem \ref{bl0} is a limiting case of Theorem \ref{bl} for $W$ approaching a delta function. 

For a generalization of Theorems \ref{bll} and \ref{bl} to more than two functions see \cite{Ch0,Dr,Mo,BuSc,BuHa}.

\subsubsection*{Application 4: Existence of minimizers}

Lieb has used symmetrization techniques in various settings to prove the existence of optimizers in variational problems. Compactness methods in Lieb's work are reviewed in detail in \cite{Sa}, so we will be brief. 

The set-up is that one wants to minimize a functional $\mathcal E$ on a certain set $X$, which is bounded in $L^p(\R^d)$. It is assumed that for all $f\in X$ one has $f^*\in X$ and $\mathcal E[f^*]\leq \mathcal E[f]$. Therefore, in order to prove the existence of a minimum, one can choose a minimizing sequence $(f_n)$ consisting of symmetric decreasing functions; that is, $f_n^*=f_n$.

The important observation now is that, if $f\in L^p(\R^d)$ is symmetric decreasing, then
$$
f(x)^p \leq |\{ y:\ |y|< |x| \}|^{-1} \int_{|y|< |x|} f(y)^p\,dy \leq \omega_d^{-1} |x|^{-d} \|f\|_{L^p(\R^d)}^p \,,
$$
so we have the pointwise bound
$$
f(x) \leq \omega_d^{-1/p} |x|^{-d/p} \|f\|_{L^p(\R^d)}
\qquad\text{for all}\ x\in\R^d \,.
$$

Returning to the above minimizing sequence $(f_n)$, we have the bound $f_n(x)\lesssim |x|^{-d/p}$, uniformly in $n$. Therefore, we can apply the Helly selection principle and obtain a subsequence that converges pointwise in $\R^d\setminus\{0\}$. The remaining step is to show that this pointwise limit is not identically equal to zero. Once this is shown, it follows typically by rather soft arguments that the pointwise limit is a minimizer.

In many applications there is a $\tilde p\neq p$ such that $\|f_n\|_{L^{\tilde p}(\R^d)}$ remains bounded for a minimizing sequence. Then one also obtains the bound  $f_n(x)\lesssim |x|^{-d/\tilde p}$, uniformly in $n$. Using the two bounds on $f_n$ one can typically use dominated convergence and conclude that the pointwise limit is not identically equal to zero.

The above strategy was first carried out by Lieb in \cite{LCho} for the Choquard (or Pekar) functional
$$
\mathcal E[u] := \int_{\R^3} |\nabla u|^2\,dx - \iint_{\R^3\times\R^3} \frac{|u(x)|^2\,|u(y)|^2}{|x-y|}\,dx\,dy
$$
with $X= \{ u\in H^1(\R^d):\ \|u\|_{L^2(\R^d)}=1 \}$. It was also used by Lieb and Oxford \cite{LiOx} for
$$
\mathcal E[\rho] := - \|\rho\|_{L^p(\R^d)}^{-2+\theta} \iint_{\R^d\times\R^d} \frac{\rho(x)\,\rho(y)}{|x-y|^\lambda}\,dx\,dy 
$$
with $X=\{ \rho \in L^1\cap L^p(\R^d):\ \rho\geq 0 \,,\ \|\rho\|_{L^1(\R^d)}=1 \}$. Here $0<\lambda<d$, $p<(d+\lambda)/d$ and $\theta = ((2d-\lambda)p - 2d)/(d(p-1))$. (Actually, in \cite{LiOx} only the case $d=3$, $\lambda=1$ and $p=4/3$ is considered, but the argument generalizes immediately to the regime mentioned above.) Similar techniques occur in \cite[Theorem 2.5]{LHLS} and \cite{LiYa}.

We emphasize that symmetrization techniques typically prove the existence of a \emph{certain} minimizing sequence that converges; they do \emph{not} show that all minimizing sequences converge (up to symmetries). The latter property is also important in some applications. Lieb has developed other powerful techniques for that purpose; see \cite{Sa}.

\subsubsection*{Application 5: An isoperimetric inequality for the heat kernel}

A special case of Theorem \ref{bll}, which is due to \cite{FrLu} and motivated the general case, states that, with the notation $x_0=x_M$,
$$
\int_{\R^M} \prod_{m=1}^M f_m(x_m) h_m(x_m-x_{m-1}) \,dx \leq \int_{\R^M} \prod_{m=1}^M f_m^*(x) h_m^*(x_m-x_{m-1}) \,dx \,.
$$
This inequality was used by Luttinger to derive various isoperimetric inequalities. We now present one inequality in this spirit. If $\Omega\subset\R^d$ is a measurable set and $V$ is a real function on $\Omega$ satisfying $|\{ V < \sigma \}|<\infty$ for all $\sigma\in\R$, then we define a symmetric increasing function $V_*$ on $\Omega^*$ as follows. (Here $\Omega^*=\R^d$ if $\Omega$ has infinite measure; if $\Omega$ has finite measure, the assumption on $V$ is trivially satisfied.) Let $\phi:\R\to (0,\infty)$ be a continuous, strictly decreasing function (e.g., $\phi(v)=e^{-tv}$ for some $t>0$). Then, by assumption on $V$, $\phi(V)\1_\Omega$ can be considered as a function on $\R^d$ that vanishes at infinity, so its rearrangement is well defined and we can define a function $V_*$ on $\Omega^*$ by
$$
(\phi(V) \1_\Omega)^* = \phi(V_*)\1_{\Omega^*} \,.
$$
This function is independent of the choice of $\phi$.

In the following, for an open set $\Omega\subset\R^d$, $-\Delta_\Omega$ denotes the Laplacian with Dirichlet boundary conditions in $L^2(\Omega)$. If $V$ satisfies the above condition and is bounded from below, it is well known that $-\Delta_\Omega +V$ can be defined as a selfadjoint operator in $L^2(\Omega)$ and that this operator has a compact resolvent. One obtains the following rearrangement inequality for its heat trace.

\begin{corollary}\label{heat}
	Let $\Omega\subset\R^d$ be an open set and let $V$ be a real, bounded below function on $\Omega$ satisfying $|\{ V < \sigma \}|<\infty$ for all $\sigma\in\R$. Then, for all $t>0$,
	$$
	\Tr e^{-t(-\Delta_\Omega+V)} \leq \Tr e^{-t(-\Delta_{\Omega^*}+V_*)} \,.
	$$
\end{corollary}

For a classification of the cases of equality in the corollary for $|\Omega|<\infty$ see \cite[Corollary 3.16]{Mo}.

\begin{proof}
	By the maximum principle, we have, for all $t>0$,
	$$
	e^{t\Delta_\Omega}(x,y) \leq \1_\Omega(x) e^{t\Delta}(x,y) \1_\Omega(y)
	\qquad\text{for all}\ x,y\in\Omega \,.
	$$
	Writing this in terms of the whole-space heat kernel $k_t(x-y) = e^{t\Delta}(x,y)$, which is a symmetric decreasing function of $|x-y|$, and applying the BLL theorem, we obtain, for all $t>0$ and all $n\in\N$,
	\begin{align*}
		\Tr \left( e^{-\frac tn V} e^{\frac tn\Delta_\Omega} \right)^n
		& \leq \Tr \left( e^{-\frac tn V}\1_\Omega e^{\frac tn\Delta} \right)^n \\
		& = \int_{\R^{d}}\cdots \int_{\R^d} \prod_{m=1}^n e^{-\frac tnV(x_m)}\1_\Omega(x_m) k_\frac tn(x_m-x_{m-1}) \,dx_1\ldots\,dx_n \\
		& \leq \int_{\R^{d}}\cdots \int_{\R^d} \prod_{m=1}^n e^{-\frac tnV_*(x_m)}\1_{\Omega^*}(x_m) k_\frac tn(x_m-x_{m-1}) \,dx_1\ldots\,dx_n \\
		& = \Tr \left( e^{-\frac tn V_*}\1_{\Omega^*} e^{\frac tn\Delta} \right)^n \,.
	\end{align*} 
	By Trotter's product formula, the left side converges to $\Tr e^{-t(-\Delta_\Omega+V)}$ as $n\to\infty$, while the right side converges, using in addition the smoothness of the boundary of $\Omega^*$, to $\Tr e^{-t(-\Delta_{\Omega^*}+V_*)}$.
\end{proof}

Luttinger \cite{Lu} derives several interesting inequalities from Corollary \ref{heat}. First, letting $t\to\infty$, one obtains for the lowest eigenvalues the inequality
$$
E(-\Delta_\Omega+V) \geq E(-\Delta_{\Omega^*} +V_*) \,.
$$
For $V=0$, this is the Faber--Krahn inequality, which proves Rayleigh's conjecture. Still in the case $V=0$, the heat trace has the asymptotics
$$
\Tr e^{t\Delta_\Omega} = (4\pi t)^{-d/2} \left( |\Omega| - (\pi t/4)^{1/2} \per\Omega + o(t^{1/2}) \right);
$$
see \cite{Ka} and, for Lipschitz domains, \cite{Br}. Thus, Corollary \ref{heat} implies the isoperimetric inequality $\per\Omega\geq\per\Omega^*$. For the derivation of further isoperimetric inequalities, for instance, for the capacity and the scattering length, we refer to \cite{Lu}.

\subsubsection*{Application 6: The sharp constant in Young's inequality}

Young's convolution inequality bounds the $L^p$-norm of a convolution of two functions in terms of the $L^p$-norms of the two functions. In some applications, for instance to Nelson's hypercontractivity theorem, it is important to know the sharp constant in this inequality.

This sharp constant was determined by Beckner \cite{Be} and Brascamp and Lieb \cite{BrLi}. The latter paper also characterizes the cases of equality. More precisely, the following theorem is proved there.

\begin{theorem}\label{young}
	Let $1\leq p,q,r\leq\infty$ with $1/p +1/q+1/r=2$. Then, for all $f\in L^p(\R^d)$, $g\in L^q(\R^d)$, $h\in L^r(\R^d)$,
	$$
	\left| \iint_{\R^d\times\R^d} f(x)g(x-y)h(y)\,dx\,dy \right| \leq (C_p C_q C_r)^d \, \|f\|_{L^p(\R^d)} \|g\|_{L^q(\R^d)} \|h\|_{L^r(\R^d)} \,,
	$$
	where $C_s = \left( s^{1/s} / s^{1/s'} \right)^{1/2}$ for $1<s<\infty$ and $C_s=1$ for $s\in\{1,\infty\}$. If $1<p,q,r<\infty$, then equality holds if and only if there are $A,B,C\in\C$, $a,b,c,k\in\R^d$ and a positive definite $J\in\R^{d\times d}$ such that, for a.e.\ $x\in\R^d$,
	\begin{align*}
		f(x) = A\, \exp( -p'(x-a,J(x-a))+ik\cdot x) \,, \\
		g(x) = B\, \exp(-q'(x-b,J(x-b)) - ik\cdot x) \,, \\
		h(x) = C\, \exp(-r'(x-c,J(x-c)) + ik\cdot x) \,.
	\end{align*}	
\end{theorem}

In fact, Brascamp and Lieb \cite{BrLi} prove a much more general family of inequalities, whose optimal constants are achieved for Gaussians and which came to be known as Brascamp--Lieb inequalities. This family has found important applications in several areas of mathematics and we refer to \cite{Zh} for details on one such application.

The basic strategy in the proof of Theorem \ref{young} is to duplicate the variables a large number $M$ of times and then to apply the BLL theorem in $\R^{2dM}$. The proof makes use of the fact that, to quote from \cite{BrLi}, ``for large $M$, all spherically symmetric, decreasing functions look like Gaussians in some sense''. For details of the proof, as well as for a sharp reversed inequality for nonnegative functions with $0<p,q,r<1$ we refer to \cite{BrLi}. For alternative proofs of Theorem \ref{young}, as well as generalizations, see \cite{LiGau,Bar0,Bar,Bar1,CaLiLo,BCCT}.

\medskip

Let us mention one more application of the BLL theorem due to Lenzmann and Sok \cite{LeSo}. They consider the map $f\mapsto \mathcal F^{-1}((\mathcal F f)^*)$, where $\mathcal F$ is the Fourier transform, and deduce from Theorem \ref{bll} that $\|f\|_{L^p(\R^d)} \lesssim \| \mathcal F^{-1}((\mathcal F f)^*) \|_{L^p(\R^d)}$, provided that $p$ is an even integer or $+\infty$. They use this fact to deduce symmetry results for ground states of higher order PDEs. The paper \cite{LeSo} has also some subtle results concerning cases of equality and makes a connection to classical works by Hardy and Littlewood.

%%%%%%%%%%%%%%%%%%%%%%%%

\subsection{Strict rearrangement inequality}

An important special case of the Riesz rearrangement inequality occurs when one of the functions ($g$, say) is \emph{strictly} symmetric decreasing; that is, $g(x)>g(y)$ if $|x|<|y|$. 

\begin{theorem}\label{strict}
	Let $f,g,h$ be nonnegative functions on $\R^d$, vanishing at infinity, and assume that $g$ is strictly symmetric decreasing. If there is equality in \eqref{eq:riesz}, then there is an $a\in\R^d$ such that, for a.e. $x\in\R^d$, $f(x)=f^*(x-a)$ and $h(x)=h^*(x-a)$.
\end{theorem}

This is due to Lieb \cite{LCho}. For an alternative proof, based on the characterization of cases of equality in the Brunn--Minkowski theorem, see \cite{Cama}.

The characterization of cases of equality in the Riesz rearrangement inequality in the case where all three functions are characteristic functions of sets of finite measure is due to Burchard \cite{Bu}. As far as we know, the equality cases in the BLL theorem for characteristic functions are only understood in certain cases; see \cite{ChFl,Ch1,ChON}.

Let us give an application, due to Lieb \cite{LHLS}, of his Theorem \ref{strict}.

\subsubsection*{Application: The sharp constant in the Hardy--Littlewood--Sobolev inequality}

The Har\-dy--Littlewood--Sobolev (HLS) inequality is the generalization of Young's inequality where one function ($g$, say) is an inverse power of the absolute value, so $g(x) = |x|^{-\lambda}$. This function does not lie in any $L^q(\R^d)$, but it only barely fails to do so for $q=d/\lambda$. The HLS inequality shows that a bound similar to Young's inequality remains valid, provided that $p,r$ avoid the endpoints $1$ and $\infty$.

The relevant quantity $\iint_{\R^d\times\R^d} f(x) |x-y|^{-\lambda} h(y)\,dx\,dy$ appears in several applications in physics. For instance, in the case $\lambda=1$ in $d=3$ it has the interpretation as a Coulomb interaction between charge densities $f$ and $h$. For some of these applications, it is useful to have good constants in the HLS inequality.

In the special case $p=r$, Lieb \cite{LHLS} was able to compute the sharp constant and to characterize all optimizers. He proved the following result.

\begin{theorem}\label{hls}
	Let $0<\lambda<d$ and $p=2d/(2d-\lambda)$. Then, for all $f,h\in L^p(\R^d)$,
	$$
	\left| \iint_{\R^d\times\R^d} f(x) |x-y|^{-\lambda} h(y)\,dx\,dy \right| \leq C_{\lambda,d} \, \|f\|_{L^p(\R^d)} \|h\|_{L^p(\R^d)} \,,
	$$
	where
	$$
	C_{\lambda,d} = \pi^{\frac \lambda2} \frac{\Gamma(\frac{d-\lambda}2)}{\Gamma(d-\frac\lambda2)} \left( \frac{\Gamma(d)}{\Gamma(\frac d2)} \right)^{1-\frac \lambda d} .
	$$
	Equality holds if and only if there are $A,B\in\C$, $a\in\R^d$ and $\gamma\in (0,\infty)$ such that, for a.e.\ $x\in\R^d$,
	\begin{align*}
		f(x) & = A (\gamma^2 + |x-a|^2)^{-(2d-\lambda)/2} \,, \\
		h(x) & = B (\gamma^2 + |x-a|^2)^{-(2d-\lambda)/2} \,.
	\end{align*}
\end{theorem}

Let us briefly comment on the proof. The first problem that Lieb solves in \cite{LHLS} is to prove the existence of an optimizer for the inequality in Theorem \ref{hls} (and even in its generalization with two different exponents $p\neq r$). This is considerably more difficult than for the minimization problems discussed in the previous subsection (Application 4). One of the reasons is that there is no second exponent $\tilde p$ to take advantage of. Also, even after having shown that the pointwise limit is not identically equal to zero, it is not obvious that it is a minimizer. It is in this connection that Lieb derived what became known as Brezis--Lieb lemma and applied what has been called the `method of the missing mass'. For an exposition of this proof we refer to \cite{Sa}.

Having shown the existence of optimizers, the task is to identify them. At this point the above-mentioned \emph{strict} rearrangement inequality comes in. Moreover, Lieb has the ingenious insight that the corresponding optimization problem has a `hidden symmetry', namely it is conformally invariant. The proof revolves around understanding the interplay between rearrangement and conformal rotation. As a consequence of the strict rearrangement inequality, any optimizer needs to be strictly symmetric decreasing and, as a consequence of the conformal invariance, optimizers can be rotated in the stereographically transformed version of the problem. Lieb shows that the functions given in the theorem are the only ones compatible with both the operation of rearrangement and that of conformal rotation. This point of view is further expanded and generalized in \cite{CaLo} and follow-up papers.

%%%%%%%%%%%%%%%%%%%%%%%%%%%%

\subsection{Beyond symmetrization}

We hope that the material in this section has demonstrated the power of symmetrization techniques and Lieb's masterful use of them. However, there are a number of optimization problems where such techniques do not seem to be useful. One class of examples comes from optimization problems for non-scalar quantities (vector fields, spinor fields, etc.). A further example is the HLS inequality on the Heisenberg group. There the sharp constant could be determined by other techniques; see \cite{FrLi4} and, for a translation of these techniques to the Euclidean case, see \cite{FrLi3}. Another symmetri\-zation-free route to the sharp HLS inequality is in \cite{FrLi1,FrLi2}. It is based on reflection positivity, a technique that Lieb has used with great success in problems, for instance, in statistical mechanics, as discussed in \cite{BjUe,FeGi,Ti}. Among further non-symmetrization techniques that have proved useful in determining sharp constants in functional inequalities we mention optimal transport (see, e.g., \cite{Bar1,CENV}) and nonlinear flow methods (see, e.g., \cite{CaLiLo,CaCaLo,BCCT,DoEsLo}).

%%%%%%%%%%%%%%%%%%%%%%%%%%%%

\section{Rearrangement in fractional Sobolev spaces}\label{sec:fracsob}

So far, we have discussed rearrangement for functions without regularity assumptions. In this and the next section we discuss rearrangement in Sobolev spaces. We begin with the case of fractional Sobolev spaces, which is, in some sense, closer to the material of the previous section.

The following rearrangement inequality slightly generalizes \cite[Corollary 2.3]{AlLi}.

\begin{theorem}\label{albl}
	Let $j$ be a nonnegative, convex function on $[0,\infty)$ with $j(0)=0$. Then, for any nonnegative functions $u,v,W$ on $\R^d$, vanishing at infinity, and $a,b\in\R\setminus\{0\}$,
	\begin{align*}
		& \iint_{\R^d\times\R^d} j(|u(x)-v(y)|) \,W(ax+by) \,dx\,dy \\
		& \quad \geq \iint_{\R^d\times\R^d} j(|u^*(x)-v^*(y)|) \,W^*(ax+by)\,dx\,dy \,.
	\end{align*}
\end{theorem}

\begin{proof}
	Since $j$ is nondecreasing, we have $j(|u(x)-v(y)|)\geq j(||u(x)|-|v(y)||)$, so it suffices to consider nonnegative $u,v$. Moreover, by considering $\min\{(W-\epsilon)_+,M\}$ and using monotone convergence, we may assume that $W\in L^1(\R^d)$. In this case, the theorem is a consequence of Theorem \ref{bl} with $F(u,v):=j(u)+j(v)-j(|u-v|)$.
\end{proof}

In particular for $j(t)=t^p$, $W(x)=|x|^{-d-sp}$ and $a=-b=1$, we obtain the following rearrangement inequality for the seminorm in the fractional Sobolev space $\dot W^{s,p}(\R^d)$.

\begin{corollary}\label{fracrearr}
	Let $1\leq p<\infty$ and $0<s<1$. Then, for any nonnegative function $u$ on $\R^d$, vanishing at infinity,
	$$
	\iint_{\R^d\times\R^d} \frac{|u(x)-u(y)|^p}{|x-y|^{d+sp}}\,dx\,dy \geq \iint_{\R^d\times\R^d} \frac{|u^*(x)-u^*(y)|^p}{|x-y|^{d+sp}}\,dx\,dy \,.
	$$
\end{corollary}

To recapitulate, Corollary \ref{fracrearr} follows via Theorem \ref{albl} from Theorem \ref{bl}, which, in turn, is a consequence of the Riesz rearrangement inequality. The underlying mechanism is most clearly seen for $p=2$.

\begin{proof}[Proof of Corollary \ref{fracrearr} for $p=2$]
	By monotone convergence, it suffices to prove the inequality with $W_\epsilon(z) := (|z|^2+\epsilon^2)^{-(d+2s)/2}$ for $\epsilon>0$, so that $W_\epsilon\in L^1(\R^d)$. Then
	\begin{align*}
		\iint_{\R^d\times\R^d} |u(x)-u(y)|^2 W_\epsilon(x-y) \,dx\,dy & = 2 \int_{\R^d} |u(x)|^2 \,dx \int_{\R^d} W_\epsilon(z)\,dz \\
		& \quad - 2\re \iint_{\R^d\times\R^d} \overline{u(x)} W_\epsilon(x-y) u(y)\,dx\,dy \,.
	\end{align*}
	The first term on the right side is invariant under rearrangement and, by Riesz, the double integral in the second term does not go down.
\end{proof}

The fractional perimeter of a measurable set $A\subset\R^d$ is defined by
$$
\per_s A := \int_A \int_{\R^d\setminus A} \frac{dx\,dy}{|x-y|^{d+sp}} \,.
$$
Since $\per_s A = (1/2) \iint_{\R^d\times\R^d} |\1_A(x)-\1_A(y)| |x-y|^{-d-s}\,dx\,dy$, we deduce from Co\-rollary \ref{fracrearr} that, if $A$ has finite measure,
\begin{equation}
	\label{eq:fracisop}
	\per_s A \geq \per_s A^* \,.
\end{equation}
This is known as the `fractional isoperimetric inequality'.

One can show \cite{Da} that $\lim_{s\to 1_-} (1-s) \per_s E= c_d \per E$ for some constant $c_d\in(0,\infty)$, and therefore from \eqref{eq:fracisop} one recovers the usual isoperimetric inequality. In some sense, the parameter $s<1$ in \eqref{eq:fracisop} plays a similar role as the parameter $\epsilon>0$ in the proof of Corollary \ref{isoper}.

The cases of equality in Corollary \ref{fracrearr} and, thus in \eqref{eq:fracisop}, were characterized in \cite{FrSe}.

We end this section with another theorem of Almgren and Lieb, whose importance will probably become clearer in the next section. It concerns continuity of the rearrangement map $u\mapsto u^*$. According to Corollary \ref{fracrearr}, this map is Lipschitz continuous at the zero function. Since it is nonlinear, its boundedness at the zero function does not imply continuity. We do know that it is continuous on $L^p(\R^d)$, $1\leq p<\infty$, as a consequence of the first inequality in \eqref{eq:expand}. The following theorem shows continuity in the fractional Sobolev space $W^{s,p}(\R^d)$, which, by definition, consists of all $u\in L^p(\R^d)$ for which the quantity on the left side of the inequality in Corollary \ref{fracrearr} is finite.

\begin{theorem}\label{alfrac}
	Let $d\geq 1$, $1\leq p<\infty$ and $0<s<1$. Then the rearrangement map $u\mapsto u^*$ is continuous on $W^{s,p}(\R^d)$.
\end{theorem}

This is \cite[Theorem 9.2]{AlLi}.

%%%%%%%%%%%%%%%%%%%%%%%%%%%%

\section{Rearrangement in Sobolev spaces}\label{sec:sob}

\subsection{Rearrangement inequalities for the gradient}

For $1\leq p\leq\infty$ we denote by $\dot W^{1,p}(\R^d)$ the set of $u\in L^1_\loc(\R^d)$ whose distributional gradient belongs to $L^p(\R^d)$.

\begin{theorem}\label{rearrgrad}
	Let $d\geq 1$ and $1\leq p\leq\infty$. If $u\in\dot W^{1,p}(\R^d)$ vanishes at infinity, then $u^*\in\dot W^{1,p}(\R^d)$ and
	$$
	\|\nabla u^* \|_{L^p(\R^d)} \leq \|\nabla u\|_{L^p(\R^d)} \,.
	$$
\end{theorem}

Results of this flavor go back to Faber and Krahn in their proof of Rayleigh's conjecture, namely that the lowest eigenvalue of the Dirichlet Laplacian among all open sets of a given measure is minimal for a ball. The technique of symmetrization was applied to other optimization problems of isoperimetric nature by P\'olya and Szeg\H{o} \cite{PoSz}. Those concern, for instance, the capacity or the torsional rigidity. We have briefly discussed some of these inequalities in Section \ref{sec:bll} above (Application 5).

The fact that in these applications the rearrangement inequality is applied to optimizers, which are typically smooth, is probably a reason why originally not much attention was paid to the precise regularity assumptions under which they are valid. The assumptions in Theorem \ref{rearrgrad} are natural and rather minimal. Results in that direction appeared in the late 1960s and 1970s, sometimes connected with the problem of finding the optimal constant in the Sobolev inequality \cite{Ro,Au,Ta}. In \cite{Hi}, Theorem \ref{rearrgrad} is proved under the additional assumption that $u\in L^p(\R^d)$ (which can be easily removed by considering $\min\{(|u|-\epsilon)_+,M\}$). These proofs typically use a substantial amount of geometric measure theory.

In \cite{LCho}, Lieb gave a short proof of Theorem \ref{rearrgrad} in the special case $p=2$, based on Riesz's rearrangement inequality and avoiding any geometric measure theory. We emphasize that this case is by far the most important one in applications. It appears, for instance, in the above-mentioned isoperimetric inequalities. We sketch this argument.

\begin{proof}[Proof of Theorem \ref{rearrgrad} for $p=2$]
	First, we notice that it suffices to prove the theorem under the additional assumption that $u\in L^2(\R^d)$. Indeed, otherwise we apply the inequality to $\min\{(|u|-\epsilon)_+,M\}$. By the chain rule for Sobolev functions, this function belongs to $\dot W^{1,2}(\R^d)$ and the length of its gradient is pointwise bounded by $|\nabla u|$. By the vani\-shing-at-infinity assumption, it belongs to $L^2(\R^d)$. Applying the rearrangement inequality for $H^1$-functions to this function and using monotone convergence as $\epsilon\to 0^+$ and $M\to\infty$, we obtain the claimed result for $u$.
	
	For $u\in L^2(\R^d)$, one easily verifies, for instance using Fourier transforms, that $\nabla u\in L^2(\R^d)$ if and only if 
	$$
	\limsup_{t\to 0^+} t^{-1} \left( \|u\|_{L^2(\R^d)}^2 - (u,e^{t\Delta}u) \right)<\infty
	$$
	and, in this case,
	$$
	\lim_{t\to 0} t^{-1} \left( \|u\|_{L^2(\R^d)}^2 - (u,e^{t\Delta}u) \right) = \|\nabla u\|^2_{L^2(\R^d)} \,.
	$$
	
	Since $e^{t\Delta}$ acts as convolution with a symmetric decreasing function, it follows from Riesz's inequality \eqref{eq:riesz} that $(u,e^{t\Delta}u)\leq (u^*,e^{t\Delta}u^*)$. The claimed rearrangement inequality for gradients now follows immediately from the above characterization of $H^1(\R^d)$.
\end{proof}

Let us make three remarks concerning Theorem \ref{rearrgrad}. First, there is another, relatively elementary proof of this theorem in \cite[Theorem 3.20]{Ba}. Second, there is a proof of the isoperimetric inequality that has some similarities with the above argument \cite{Le}. Third, since there is a limiting argument involved, the above proof does not give any information about the cases of equality in Theorem \ref{rearrgrad}. This is, indeed, a rather subtle question; see, e.g., \cite{BrZi,CiFu}.

We also would like to stress that, apart from its use to solve certain variational problems (meaning, for instance, finding optimal constants), Theorem \ref{rearrgrad} is also useful to prove the existence of minimizers for variational problems, similarly as in Section~\ref{sec:bll} (Application 4).

%%%%%%%%%%%%%%%%%%%%%

\subsection{(Non)continuity of rearrangement in $W^{1,p}$}

We now present some deep results by Almgren and Lieb \cite{AlLi} concerning the continuity of the rearrangement map $u\mapsto u^*$; see also the summaries \cite{AlLi1,AlLi2}. Theorem \ref{rearrgrad} shows that this map is (Lipschitz) continuous at the zero function but, since it is nonlinear, this does not imply continuity in general.

Before the work of Almgren and Lieb, Coron \cite{Co} had proved that, in one dimension, the rearrangement map is continuous in $W^{1,p}(\R)$ for any $1<p<\infty$. Therefore it came as quite a surprise when Almgren and Lieb showed that in dimensions $d\geq 2$ the map is, in general, not continuous in $W^{1,p}(\R^d)$ for any $1\leq p<\infty$. Moreover, they showed not only this, but even characterized the points where the map is continuous.

Let us describe the results of Almgren and Lieb. Following them, we denote by $\mathcal C_1$ the set of nonnegative functions $u$ on $\R^d$ that vanish at infinity and whose distributional gradient belongs to $L^1(A)$ for any set $A\subset\R^d$ of finite measure. For $u\in\mathcal C_1$ we introduce the \emph{residual distribution function}
$$
\mathcal G_u(\tau) := |\{ u>\tau \}\cap\{ \nabla u = 0\}|
\qquad\text{for}\ \tau>0 \,.
$$
This defines a nonincreasing function and so its distributional derivative is a nonpositive Radon measure on $(0,\infty)$. A function $u\in\mathcal C_1$ is called \emph{co-area regular} if the absolutely continuous part (with respect to Lebesgue measure) of this measure vanishes. Otherwise, it is called \emph{co-area irregular}.

Intuitive examples of functions with flat spots have a residual distribution function with jump discontinuities. Quite surprisingly, Almgren and Lieb showed that, in $d\geq 2$, there are functions with derivatives vanishing on sets of positive Lebesgue measure, but for which the residual distribution function is absolutely continuous. This will be discussed below. The connection between co-area regularity and continuity of the rearrangement map is as follows.

\begin{theorem}\label{alli}
	Let $d\geq 1$ and $1\leq p<\infty$. Then the rearrangement map $u\mapsto u^*$ is continuous in $W^{1,p}(\R^d)$ precisely at those $u$ such that $|u|$ is co-area regular.
\end{theorem}

This theorem motivates one to study the class of co-area regular functions. The following result says that sufficiently smooth functions are co-area regular; see \cite[Theorems 5.2 and 5.4]{AlLi}

\begin{proposition}\label{coareareg}
	If $d=1$, then any function $u\in\mathcal C_1$ is co-area regular. If $d\geq 2$ and $u\in\mathcal C_1 \cap C_\loc^{d-1,1}(\R^d)$, then $u$ is co-area regular.
\end{proposition}

The $C_\loc^{d-1,1}$ condition for $d\geq 2$ means that the function is $d-1$ times continuously differentiable and the $(d-1)$-st derivative is locally Lipschitz continuous. The conclusion in this case is essentially a consequence of the Morse--Sard--Federer theorem. Note that, since all $\mathcal C_1$ functions in dimension one are co-area regular, Theorem \ref{alli} recovers Coron's continuity result.

Remarkably, in dimension $d\geq 2$ there are co-area irregular functions and, more\-over, the regularity requirement in Proposition \ref{coareareg} is optimal on the H\"older scale; see \cite[Subsection 5.1]{AlLi}.

\begin{proposition}\label{coareairreg}
	Let $d\geq 2$, $0<\lambda<1$ and $0<\alpha<1$. Then there is an $f\in C^{d-1,\lambda}(\R^d)$ such that $0\leq f\leq 1$, $\supp f \subset [-1,1]^d$ and
	$$
	\mathcal G_f(\tau) = 2^d (1-\alpha) (1-\tau)_+
	\qquad\text{for all}\ \tau>0 \,.
	$$
	In particular, $f$ is co-area irregular.
\end{proposition}

In fact, one can even prove that co-area irregular functions are dense if $d\geq 2$; see \cite[Theorem 5.3]{AlLi}.

Theorem \ref{alli} should be contrasted with Theorem \ref{alfrac}, which says that the map $u\mapsto u^*$ is continuous in $W^{s,p}(\R^d)$ for all $d\geq 1$, $0<s<1$ and $1\leq p<\infty$. We also mention that Steiner symmetrization is continuous in $W^{1,p}(\R^d)$ for $d\geq 2$ \cite{Bu2}.

%%%%%%%%%%%%%%%%%%%%%%%%%%%%

\section{Stability questions}\label{sec:stab}

When it comes to sharp functional inequalities, the basic questions are to determine the sharp constant and to characterize the case of equality. A natural next step is to investigate the stability, that is, whether an almost optimal value of the functional in question implies that the state is close, in some distance, to an optimizer.

A qualitative version of this question can often be answered affirmatively using compactness methods, for instance, those presented in \cite{Sa}. It is essentially the question whether all minimizing sequences are relatively compact (up to symmetries).

A quantitative version of the question asks whether it is possible to bound the distance to an optimizer from above by a power of the difference of the value of the functional and its optimal value.

In the paradigmatic setting of the Sobolev inequality in $\dot H^1(\R^d)$, this question was asked in the influential paper \cite{BreLi} by Brezis and Lieb. An affirmative answer was given in \cite{BiEg}. Generalizations of this question have received a lot of attention, in particular, in the past fifteen years. For instance, a quantitative version of the isoperimetric inequality appeared in \cite{FMP} and quantitative versions of Theorem \ref{rearrgrad} and Corollary \ref{fracrearr} in \cite{CEFT,FFMMM}. For further references, we refer to \cite{Fr}.

Here we would like to focus on the Riesz rearrangement inequality. As we have already mentioned, Burchard \cite{Bu} characterized the cases of equality in the situation where all three functions are characteristic functions. In a deep paper \cite{Ch2} Christ has derived a quantitative version of this result. Similar results for particular cases of the BLL inequalities appear in \cite{Ch1,ChON}. A quantitative version of Lieb's strict rearrangement inequality of a somewhat different flavor can be found in \cite{Cama}.

To motivate the following result, let $B\subset\R^d$ be a ball, centered at the origin and note that, by the Riesz rearrangement inequality, for any set $E\subset\R^d$ of finite measure,
\begin{equation}
	\label{eq:rieszsimple}
	\iint_{E\times E} \1_B(x-y)\,dx\,dy \leq \iint_{E^*\times E^*} \1_B(x-y)\,dx\,dy \,.
\end{equation}
In applications one is also interested in a certain generalization of this, which is obtained by the bathtub principle. Namely, if $\rho\in L^1(\R^d)$ with $0\leq\rho\leq 1$, then
\begin{equation}
	\label{eq:rieszbath}
	\iint_{\R^d\times\R^d} \rho(x)\1_B(x-y)\rho(y)\,dx\,dy \leq \iint_{E^*\times E^*} \1_B(x-y)\,dx\,dy  \,,
\end{equation}
where $E^*$ is a ball of measure $|E^*|=\int_{\R^d} \rho\,dx$. By \cite{Bu}, under the condition $|B|^{1/d}<2|E|^{1/d}$, equality holds in \eqref{eq:rieszsimple} if and only if $E$ is (equivalent to) a translate of a ball of measure $|E|$. From this one deduces that, under the condition $|B|^{1/d}<2 \|\rho\|_{L^1(\R^d)}^{1/d}$, equality holds in \eqref{eq:rieszbath} if and only if $\rho$ is a.e.\ equal to the characteristic function of a ball of measure $\|\rho\|_{L^1(\R^d)}$.

The question now is whether the difference between the two sides of \eqref{eq:rieszbath} controls the distance of $\rho$ from translates of characteristic functions of balls of the appropriate measure. The natural measure of distance is the $L^1$-norm. The following result is from \cite{FrLi}; see also \cite{FrLi6}. It is obtained by a slight modification of Christ's methods in \cite{Ch2}. Indeed, for $\rho$ equal to a characteristic function, it is a special case of the result in \cite{Ch2}.

\begin{theorem}\label{christ}
	Let $0<\delta\leq 1/2$. Then there is a constant $c_{d,\delta}>0$ such that, for all balls $B\subset\R^d$, centered at the origin, and all $\rho\in L^1(\R^d)$ with $0\leq\rho\leq 1$ and
	$$
	\delta \leq \frac{|B|^{1/d}}{2\, \|\rho\|_{L^1(\R^d)}^{1/d}} \leq 1-\delta \,,
	$$
	one has
	$$
	\iint_{\R^d\times\R^d} \rho(x)\1_B(x-y)\rho(y)\,dx\,dy \leq \iint_{E^*\times E^*} \1_B(x-y)\,dx\,dy - c_{d,\delta} \|\rho\|_{L^1(\R^d)}^2 A[\rho]^2 \,,
	$$
	where $E^*$ is a ball of measure $|E^*|=\int_{\R^d} \rho\,dx$ and where
	$$
	A[\rho] := \left( 2\|\rho\|_{L^1(\R^d)} \right)^{-1} \inf_{a\in\R^d} \left\| \rho-\1_{E^*+a} \right\|_{L^1(\R^d)} \,.
	$$
\end{theorem}

As a consequence we obtain the following quantitative rearrangement bounds for Riesz energies.

\begin{corollary}\label{christcor1}
	Let $0<\lambda<d$. Then there is a $c_{\lambda,d}>0$ such that, for all $\rho\in L^1(\R^d)$ with $0\leq\rho\leq 1$, one has
	$$
	\iint_{\R^d\times\R^d} \frac{\rho(x)\,\rho(y)}{|x-y|^\lambda}\,dx\,dy \leq \iint_{E^*\times E^*} \frac{dx\,dy}{|x-y|^\lambda}\,dx\,dy
	- c_{\lambda,d} \|\rho\|_{L^1(\R^d)}^{2-\lambda/d} A[\rho]^2 \,,
	$$
	where $E^*$ and $A[\rho]$ are as in Theorem \ref{christ}.
\end{corollary}

For $d=3$, $\lambda=1$, this corollary was shown earlier in \cite{BuCh}. It was shown simultaneously and independently in \cite{FuPr} in the range $0<\lambda<d-1$. Some generalizations appear in \cite{YaYa} and another proof for $0<\lambda<d-1$ in \cite{BuCh2}.

\begin{proof}
	We write
	\begin{align*}
		& \iint_{E^*\times E^*} \frac{dx\,dy}{|x-y|^\lambda} - \iint_{\R^d\times\R^d} \frac{\rho(x)\,\rho(y)}{|x-y|^\lambda}\,dx\,dy \\
		& = \lambda \int_0^\infty \frac{dR}{R^{\lambda+1}} \left( \iint_{E^*\times E^*} \1_{B_R}(x-y)\,dx\,dy -  \iint_{\R^d\times\R^d} \rho(x)\1_{B_R}(x-y)\rho(y)\,dx\,dy \right),
	\end{align*}
	where $B_R$ denotes the ball, centered at the origin, of radius $R$. Let
	$$
	I := \left\{ R>0 :\ \frac14 \leq \frac{|B_R|^{1/d}}{2\,\|\rho\|_{L^1(\R^d)}^{1/d}}\leq \frac34 \right\}.
	$$
	For $R\not\in I$, we bound the integrand of the $R$-integral from below by zero according to Riesz's theorem, while for $R\in I$ we apply Theorem \ref{christ} with $\delta=1/4$. We obtain
	\begin{align*}
		& \iint_{E^*\times E^*} \frac{dx\,dy}{|x-y|^\lambda} - \iint_{\R^d\times\R^d} \frac{\rho(x)\,\rho(y)}{|x-y|^\lambda}\,dx\,dy \geq \lambda\, c_{d,1/4}\, \|\rho\|_{L^1(\R^d)}^2\, A[\rho]^2\int_I \frac{dR}{R^{\lambda+1}} \,.
	\end{align*}
	The integral on the right side is a constant, depending on $d$ and $\lambda$, times $\|\rho\|_{L^1(\R^d)}^{-\lambda/d}$. This proves the corollary.
\end{proof}

Using a similar argument, one obtains a quantitative form of the fractional isoperimetric inequality \eqref{eq:fracisop} (albeit with a worse $s$-dependence than \cite{FFMMM}) and a quantitative rearrangement inequality for $\iint_{\R^d\times\R^d} \rho(x) |x-y|^\alpha \rho(y)\,dx\,dy$ with $\alpha>0$. The latter was used to analyze a flocking problem in \cite{FrLi}.

%%%%%%%%%%%%%%%%%%%%%%%%%%%%%%%%%%%%%%%%%%%%%%%%%%%%%%%%%%%%%%%%%%%%%%%%%%%%%%%%
%%%%%%%%%%%

\bibliographystyle{amsalpha}

\end{document}